\documentclass[12pt]{article}

\usepackage{graphicx}
\usepackage{epstopdf}
\usepackage[ruled,vlined]{algorithm2e}
\usepackage{amsmath, amssymb}
\usepackage{amsthm}

\usepackage{tikz}
\usepackage{multicol}
\usepackage{caption}				
\usetikzlibrary{positioning,shadows,backgrounds,shapes.geometric}
\usetikzlibrary{decorations.pathreplacing}

\newtheorem{theorem}{Theorem}[section]
\newtheorem{corollary}[theorem]{Corollary}
\newtheorem{lemma}[theorem]{Lemma}

\usepackage{graphicx}

\date{}

\title{\textbf{An Optimal Algorithm for Stopping on the Element Closest to the Center of an Interval \footnote{The research was partially supported by NCN Grant DEC-2015/17/B/ST6/01868.}}\\
\large{}}
\author{Ewa M. Kubicka, 
Grzegorz Kubicki,
Ma{\l}gorzata Kuchta,\\
and Ma{\l}gorzata Sulkowska}

\begin{document}
\maketitle

\begin{abstract}
Real numbers from the interval [0, 1] are randomly selected with uniform distribution.  There are $n$ of them and they are revealed one by one. However, we do not know their values but only their relative ranks. We want to stop on recently revealed number maximizing the probability that that number is closest to $\frac{1}{2}$.  We design an optimal stopping algorithm achieving our goal and prove that its probability of success is asymptotically equivalent to $\frac{1}{\sqrt{n}}\sqrt{\frac{2}{\pi}}$.\\

\begin{tabular}{lp{8cm}}
2010 \textit{Mathematics subject classification:}&{Primary 60G40} \\&{Secondary 90C27}
\end{tabular}
\end{abstract}

\section{Introduction} 
Consider the following online problem: $n$ numbers randomly selected from the interval [0, 1] are presented to us one number at a time. After revealing $k$ numbers, $1 \le k \le n$, we know their ranks but not their values. We ignore the zero probability event that some of them are equal. Our goal is to stop on the presently revealed number $x_k$ hoping that  $x_k$ is closest to $\frac{1}{2}$, the center of the interval, among all $n$ numbers.
We will construct an optimal stopping algorithm and show that this algorithm, for large values of $n$, has the probability of success of order $\frac{1}{\sqrt{n}}\sqrt{\frac{2}{\pi}}$.

This problem is a new relative of the classical secretary problem. In classical secretary problem, $n$ candidates are linearly ordered and our goal is to stop on the best candidate. In our model, if the objective were to stop on the element closest to 1, then  the problem would be equivalent to the original secretary problem. The classical secretary problem, with its solution written down by Lindley  \cite{L}  in 1961, attracted a lot of attention and has been considered in various modifications. The paper  \cite{F} provides nice and deep survey of this research. Many generalizations of classical problem are studied, for example if linear order was replaced by partial order 
 \cite{M},  \cite{P},  \cite{FW}, \cite{GKMN},  \cite{S}  or just by graph or digraph structure  \cite{KM},  \cite{GKK}, \cite{Su}, \cite{BS}. One of the versions of the secretary problem similar to our model is to stop on the element of middle rank. It was considered in \cite{R} under a variety of circumstances.
That problem is however different than the topic of this paper. Stopping on the middle rank element does not guarantee that this number would be closest to $\frac{1}{2}$. On the other hand, in our case stopping on elements other than the middle one gives nonzero probability of success.
 
 The paper is organized as follows. In Section \ref{sec_optimal}, we construct an optimal stopping algorithm and justify a formula for its probability of success. We use backwards induction to do so. Not surprisingly, the algorithm tells us to stop late and  only on the numbers having ranks not far from the middle. We provide an example how this algorithm works for $n = 10$ and how the stopping region looks like.
 The asymptotic performance of the algorithm is analyzed in Section \ref{sec_asymptotics}. First we contract an algorithm that is not optimal but has more regular stopping region which allows us to estimate the asymptotic performance of our algorithm from below. Then we consider a little easier online problem in which the optimal strategy has the asymptotic performance that is easy to calculate and provides an upper bound for our algorithm. It happens that those bounds are identical just proving that the asymptotic probability of success of the optimal stopping algorithm is of order $\sqrt{\frac{2}{\pi}}\frac{1}{\sqrt{n}}$.
 
 \section{Optimal Stopping Algorithm} \label{sec_optimal}
 Assume that $n$ different numbers $x_1, x_2, ..., x_n$ from the interval [0, 1] are randomly selected, with uniform distribution, and presented to us one by one. We know $n$ in advance but after revealing $t$ numbers, $1 \le t \le n$, we know only their relative ranks, not their values. Let's rename them such that, at that moment, we know their order $y_1^{(t)} < y_2^{(t)} < ... < y_t^{(t)}$ and we know that the rank of $x_t$ is $r$; it means that $x_t = y_r^{(t)}$. Our goal is to stop on the presently revealed number $x_t$ maximizing the probability that $|x_t - \frac{1}{2}| \le |x_i - \frac{1}{2}|$ for all $i$, $1 \le i \le n$, the probability that $x_t$ will be the closest to the midpoint of the interval; we will call such an event ``$x_t$ is the best''.
 
Before constructing optimal stopping algorithm (it will be denoted by $\mathcal{A}_n$), we need two results providing formulas for the probability that the number of specific rank is the best.
 
 \begin{theorem}\label{bestprob}
 If $ y_1 < y_2 < ... < y_r < ... < y_n$ are ranked numbers revealed at time $n$, then 
 $$Pr (y_r \text{ is the best}) =  \binom{n-1}{r-1}  \cdot \frac{1} {2^{n-1}}.$$
 \end{theorem}
 \begin{proof}
 We have 
 \begin{align*}
 \begin{split}
  \text{Pr}(y_r \text{ is the best}) & 
 = \text{Pr}\Big(\Big(y_{r} < {1 \over 2} < y_{r+1} \Big) \text{ and } \Big( |y_r - \frac{1}{2}| \le |y_{r+1} - \frac{1}{2}|\Big) \Big) \\
 & +  \text{Pr}\Big(\Big(y_{r-1} < \frac{1}{2} < y_{r} \Big) \text{ and }
   \Big(|y_{r-1} - \frac{1}{2}| \ge |y_{r} - \frac{1}{2}|\Big)\Big) 
\end{split}
\end{align*}
Note that
\begin{align*}
\begin{split}
\text{Pr}\Big(\Big(|y_r - \frac{1}{2}| & \le |y_{r+1} - \frac{1}{2}|\Big) \Big| \Big( y_{r} < {1 \over 2} < y_{r+1}\Big) \Big) \\
& = \text{Pr}\Big( \min\{Z_1, Z_2, \ldots, Z_r\} < \min\{Z_{r+1}, Z_{r+2}, \ldots, Z_n\} \Big) = \frac{r}{n},
\end{split}
\end{align*}
where $Z_1, Z_2, \ldots, Z_n$ are independent random variables drawn uniformly from the interval $[0,1/2]$. Analogously we get
$$
\text{Pr}\Big(\Big(|y_{r-1} - \frac{1}{2}| \ge |y_{r} - \frac{1}{2}|\Big) \Big| \Big( y_{r-1} < \frac{1}{2} < y_{r} \Big) \Big) = \frac{n-r+1}{n}
$$
and finally
\begin{align*}
\begin{split}
  \text{Pr}(y_r \text{ is the best}) & 
	 = \binom{n}{r} \cdot \frac{1}{2^n}  \cdot \frac{r}{n} +  \binom{n}{r-1} \cdot \frac{1}{2^n}  \cdot \frac{n-r+1}{n} \\
 & = \frac{1}{2^n} \bigg[ \frac{(n-1)!}{(r-1)!(n-r)!} + \frac{(n-1)!}{(r-1)!(n-r)!} \bigg] =\binom{n-1}{r-1} \cdot \frac{1}{2^{n-1}}.
\end{split}
\end{align*}
\end{proof}

\begin{theorem}\label{generalprob}
 If $ y_1^{(t)} < y_2^{(t)} < ... < y_r^{(t)} < ... < y_t^{(t)}$ are ranked numbers revealed at time $t$, then 
 \begin{align}\label{will-be-best}
 \begin{split}
	Pr (y_r^{(t)} \text{ will be the best}) =   \frac{1}{2^{n-1}} \sum\limits_{j=0}^{n-t}\binom{n-1}{r-1+j} \binom{n-t}{j}  \frac{r^j(t+1-r)^{n-t-j}}{(t+1)^{n-t}}.
 \end{split}
 \end{align}
 \end{theorem}
 \begin{proof}
 Since $n-t$ additional numbers will be revealed, the rank $r$ of the number $y_r^{(t)}$ would increase by some $j$, where $0 \le j \le n-t$. Every number following $y_r^{(t)}$ will fall independently, with the same probability 
 $\frac{1}{t+1}$, into one of the intervals (0, $y_1^{(t)}$), ($y_1^{(t)}$, $y_2^{(t)}$), ..., ($y_t^{(t)}$, 1).  Every time a number falls into one of the first $r$ intervals, the rank of $y_r^{(t)} $ is increased by 1.
 Therefore, the probability that after revealing all $n$ numbers, the rank of $y_r^{(t)} $ will be $r + j$ is $\binom{n-t}{j} \frac{r^j(t+1-r)^{n-t-j}}{(t+1)^{n-t}}.$  Then, from Theorem \ref{bestprob}, 
 $$\text{Pr }(y_r^{(t)} \text{ will be the best } \vert \text{ its rank is }r+j) = \binom{n-1}{r-1+j}  \frac{1}{2^{n-1}} $$ and the formula (\ref{will-be-best}) follows.
 \end{proof}

From now on $Pr (y_r^{(t)} \text{ will be the best})$ will be, for short, denoted by $P_r^{(t)}$. Also, by $ \mathcal{A}_n^{(t)} $, we denote the optimal algorithm that stops only in rounds $t$, $t+1, ..., n-1$, or $n$ (it never stops before time $t$). We are ready to construct an optimal stopping algorithm $\mathcal{A}_n$ using backwards induction. Note that $\mathcal{A}_n = \mathcal{A}_n^{(1)}$.

$\mathcal{A}_n^{(n)}$ is an algorithm that stops only on the number that came in the last round, thus 
$\text{Pr}( \mathcal{A}_n^{(n)} \text{ succeeds)} = \frac{1}{n}$. Algorithm $\mathcal{A}_n^{(n-1)}$ stops only in rounds $n-1$ or $n$. Therefore, it stops on number $y_r^{(n-1)}$ in the 
$(n-1)^{\text{st}}$ round if $\text{P}_r^{(n-1)} \ge \frac{1}{n}$. Using the formula from Theorem \ref{generalprob} with $t = n-1$, we get the inequality
$ \frac{1}{2^{n-1}} \big[ \binom{n-1}{r-1} \frac{n-r}{n} + \binom{n-1}{r} \frac{r}{n} \big] \ge \frac{1}{n}$ which is equivalent to 
\begin{equation*}
 \binom{n-2}{r-1} \ge \frac{2^{n-2}}{n-1}.
\end{equation*}
Solving it for $r-1$ gives a symmetric interval from the $(n-2)^{\text{nd}}$ row of the Pascal triangle, namely $r-1 \in [z_1, \text{ } n-2-z_1]$ for some $z_1$, 
or, setting $r_1=z_1+1$, $r \in [r_1, \text{ } n-r_1]$.\\
Therefore, algorithm  $\mathcal{A}_n^{(n-1)}$ stops in round $n-1$ if and only if the rank of the number that comes in that round is from the stopping interval $[r_1, \text{ } n-r_1]$.
Of course,
$$\text{Pr}( \mathcal{A}_n^{(n-1)} \text{ succeeds)} = \sum\limits_{r=r_1}^{n-r_1} \frac{1}{n-1} \text{P}_r^{(n-1)} + \frac{2(r_1-1)}{n-1} \frac{1}{n},$$
where the two terms count the probabilities of winning if the
$(n-1)^{\text{st}}$ number has the rank from $[r_1, \text{ } n-r_1]$ or from outside of that interval, respectively.

In general, assume that for $k = t+1, t+2, \ldots, n$ we know the probabilities $\text{Pr}( \mathcal{A}_n^{(k)} \text{ succeeds)}$ and the stopping region in round $k$, the interval $[r_{n-k}, \text{ } k+1-r_{n-k}]$. Then the optimal algorithm $ \mathcal{A}_n^{(t)} $ stops on the number $y_r^{(t)}$ in round $t$ if and only if its rank $r$ satisfies the inequality
\begin{equation}\label{prob_rec}
\text{P}_r^{(t)} \ge \text{Pr}(\mathcal{A}_n^{(t+1)} \text{ succeeds}).
\end{equation}
If the inequality (\ref{prob_rec}) has a solution, then the solution set, the symmetric interval $[r_{n-t}, \text{ } t+1-r_{n-t}]$, is the stopping region for $ \mathcal{A}_n^{(t)} $ in round $t$ and 
$$
\text{Pr}(\mathcal{A}_n^{(t)} \text{ succeeds}) = \sum\limits_{r=r_{n-t}}^{t+1-r_{n-t}} \frac{1}{t}\text{P}_r^{(t)} + \frac{2(r_{n-t}-1)}{t}\text{Pr}(\mathcal{A}_n^{(t+1)} \text{ succeeds}).
$$
If there is no $r$ satisfying inequality (\ref{prob_rec}), then the algorithm $ \mathcal{A}_n^{(t)} $ never stops in round $t$ and $\text{Pr}(\mathcal{A}_n^{(t)} \text{ succeeds)} = \text{Pr}(\mathcal{A}_n^{(t+1)} \text{ succeeds)}$. Recall that the optimal algorithm for our decision problem is  $ \mathcal{A}_n = \mathcal{A}_n^{(1)}$. \\

\begin{algorithm}[H]
\DontPrintSemicolon
\caption{\textbf{$\mathcal{A}_n^{(t)}$}}
\vspace{5pt}
\KwData{$x_1, x_2, \ldots, x_n$ - numbers chosen uniformly at random from~$[0,1]$; $\text{Pr}( \mathcal{A}_n^{(k)} \text{ succeeds)}$ 
for $k = t+1, t+2, ..., n$} 
\KwResult{candidate for the number being the closest one to $1/2$ among $x_1, x_2, \ldots, x_n$}
\Begin{
  \If {$t == n$} {\Return $x_n$ \;}
	\For{$j = 1, 2, \ldots, t$}{
		reveal $x_j$ \;
	}
	$r$ := rank of element $x_t$ among ordered $x_1, \ldots, x_t$ ($x_t$ = $y_r^{(t)}$) \;
	\uIf {rank $r$ satisfies $\text{P}_r^{(t)} \ge \text{Pr}(\mathcal{A}_n^{(t+1)} \text{ succeeds})$}
			{\Return $x_t$ \;}
	\Else {\Return $\mathcal{A}_n^{(t+1)}$ \;}
}
\vspace{5pt}
\footnotesize{Recall that  $y_1^{(t)} < y_2^{(t)}  < ... < y_{t}^{(t)} $ are ordered $x_1, x_2, \ldots, x_{t}$, i.e., ordered numbers revealed till round $t$.}
\end{algorithm}
\vspace{5pt}
An implementation of algorithm $\mathcal{A}_n$ is straightforward and our next example illustrates how the optimal stopping strategy looks like for $n = 10$.\\

\noindent  Example:

\begin{table}[htp]
  \begin{center}
  
    \resizebox{8.5cm}{!} {
  	\begin{tabular}{|| c | c | c | c | c ||  }  \hline  \hline
	\rule[-8pt]{0pt}{22pt}  $t$   & $10-t$ & $r_{10-t}$ & \text{stopping interval} & $Pr(\mathcal{A}_{10}^{(t)} \text{ succeeds})$    \\  \hline \hline
	 \rule[-8pt]{0pt}{20pt} 1 & 9 &   &  &  \textbf{0.1893}    \\  \hline 
	 \rule[-8pt]{0pt}{20pt} 2 &  8  &   &    &  0.1893       \\  \hline 
	  \rule[-8pt]{0pt}{20pt} 3 & 7 & 2  & $\{2\}$ &  0.1893    \\  \hline 
	   \rule[-8pt]{0pt}{20pt} 4 & 6 &   &  &  0.1858    \\  \hline 
	     \rule[-8pt]{0pt}{20pt} 5 & 5 & 3  & $\{3\}$ &  0.1858    \\  \hline 
	       \rule[-8pt]{0pt}{20pt} 6 & 4 & 3  & $[3,4]$ &  0.1798    \\  \hline 
	           \rule[-8pt]{0pt}{20pt} 7 & 3 & 3  & $[3,5] $&  0.1701    \\  \hline 
	               \rule[-8pt]{0pt}{20pt} 8 & 2 & 4  & $[4,5] $&  0.1585    \\  \hline 
	                   \rule[-8pt]{0pt}{20pt} 9 & 1 & 4  & $[4,6] $&  0.1378    \\  \hline 	
	                       \rule[-8pt]{0pt}{20pt} 10 & 0 & 1  & $[1,10] $&  0.1   \\  \hline \hline 		
		\end{tabular}
		
	}
	\vspace{5pt}
	\caption {Stopping intervals at time $t$ and probabilities that the algorithm $\mathcal{A}_{n}^{(t)}$ succeeds for $n = 10$.}
  \end{center}
\end{table}

\noindent  The optimal algorithm $\mathcal{A}_{10}$ never stops in rounds 1, 2, and 4. It stops in round 3 only on the number of the middle rank. The stopping region is shaded in Figure 1. The number in bold in Table 1 is $Pr(\mathcal{A}_{10}^{(1)} \text{ succeeds})$ which is the performance of $\mathcal{A}_{10}$.

\begin{center}
\begin{tikzpicture}[y=7.5mm,x=8.66mm] \label{stop_reg}

  \colorlet{even}{cyan!40!black}
  \colorlet{odd}{red!60!white}


  \tikzset{
    box/.style={
      regular polygon,
      regular polygon sides=6,
      minimum size=10mm,
      inner sep=0mm,
      outer sep=0mm,
      text centered,
      font=\small\bfseries\sffamily,
      text=#1!50!black,
      draw=#1,
      line width=.25mm,
      rotate=30,
    },
    link/.style={black,  shorten >=2mm, shorten <=2mm, line width=1mm},
   } 
  \node[box=even] (p-0-0) at (0,0) {\rotatebox{-30}{1}};
  \foreach \row in {1,...,9} {
    \node[box=even] (p-\row-0) at (-\row/2,-\row) {\rotatebox{-30}{1}};
    \pgfmathsetmacro{\myvalue}{1};

    \foreach \col in {1,...,\row} {

      \pgfmathtruncatemacro{\myvalue}{\myvalue + 1};
      
     \global\let\myvalue=\myvalue
      \coordinate (pos) at (-\row/2+\col,-\row);
      \pgfmathtruncatemacro{\rest}{mod(\myvalue,2)}
       \node[box=even] (p-\row-\col) at (pos) {\rotatebox{-30}{\myvalue}}; 
      
    }
}
    
   \begin{pgfonlayer}{background}
    \foreach \i/\j in {2/1,4/2,5/2,5/3,6/2,6/3,6/4,7/3,7/4,8/3,8/4,8/5,9/0, 9/1, 9/2, 9/3,9/4,9/5,9/6,9/7,9/8,9/9}
        \node[box=even,fill=odd]  at (p-\i-\j) {};
  \end{pgfonlayer}
  
   \node[right=46mm of p-0-0.center, align=left]  {$ t=1$}; 
    \node[right=42mm of p-1-1.center, align=left]  {$ t=2$}; 
     \node[right=38mm of p-2-2.center, align=left]  {$ t=3$}; 
      \node[right=34mm of p-3-3.center, align=left]  {$ t=4$}; 
 \node[right=30mm of p-4-4.center, align=left]  {$ t=5$}; 
  \node[right=26mm of p-5-5.center, align=left]  {$ t=6$}; 
     \node[right=22mm of p-6-6.center, align=left]  {$ t=7$}; 
      \node[right=18mm of p-7-7.center, align=left]  {$ t=8$}; 
         \node[right=14mm of p-8-8.center, align=left]  {$ t=9$};
          \node[right=10mm of p-9-9.center, align=left]  {$ t=10$}; 
         
\end{tikzpicture}\\
{\bf Figure 1.} The stopping region for the optimal algorithm $\mathcal{A}_{10}$.
\end{center}

As can be seen from this example, the stopping region for our algorithm $\mathcal{A}_n$ is rather irregular and the recursive formulas used to calculate $\text{Pr}(\mathcal{A}_n \text{ succeeds})$ give little hope for finding a closed formula for that probability. Despite these shortcomings, in the next section we will provide  asymptotic performance of the optimal algorithm $ \mathcal{A}_n$.

\section{Asymptotics} \label{sec_asymptotics}
Throughout this section we use standard notation:
\begin{center}
$f(n)\sim g(n)$ if $\frac{f(n)}{g(n)} \xrightarrow [n \to\infty ]{}  1$ and $f(n) = o(g(n))$ if $\frac{f(n)}{g(n)} \xrightarrow [n \to\infty ]{}  0$.\\
\end{center}
Also, the binomial coefficient $\binom{n}{z}$ for $z$ not being a natural number is understood as
$$
\binom{n}{z} = \frac{\Gamma(n+1)}{\Gamma(z+1)\Gamma(n-z+1)},
$$
where $\Gamma(z)$ is the special function gamma.

The example from the previous section for $n = 10$ might be misleading, because for large values of $n$ the stopping region of the optimal algorithm is relatively small. Based on computer simulations done for $n$ going as far as 5000, we found that, for large values of $n$, the algorithm $ \mathcal{A}_n$ does not stop until it reaches the round $\lceil n - n^{2/3}\sqrt{\ln n} \rceil$ and in the round $n - 1$ stops only on the elements whose ranks are close to the middle.  In fact, for large values of $n$, $r_1 \sim \frac{n}{2} - \frac{1}{2}\sqrt{n \ln \frac{2n}{\pi}}$ and we will prove this result in Corollary \ref{C6}. We need several simple lemmas before.

\begin{lemma}\label{L3}
For large values of $n$, $ \binom{n}{n/2} \sim \frac{\sqrt{2} \cdot 2^n}{\sqrt{\pi n}}$.
\end{lemma}
The formula follows easily from the Stirling's approximation. 

\begin{lemma}\label{L4}
If $s = s(n)$ and $w = w(n)$ are positive sequences such that\\  $s(n) \xrightarrow [n \to\infty ]{}  \infty$ and $w(n) = o(s(n))$, then $ \frac{\binom{2s}{s}}{\binom{2s}{s-w}} \sim e^{w^2/s}$.
\end{lemma}
\begin{proof}
The ratio $ \frac{\binom{2s}{s}}{\binom{2s}{s-w}}$ simplifies to
\begin{center}
$\frac{(s+1)(s+2)...(s+w-1)(s+w)}{(s-w+1)(s-w+2)...(s -1)s}=(1+\frac{w}{s-w+1})(1+\frac{w}{s-w+2})...(1+\frac{w}{s})$.
\end{center}
Therefore, $(1+\frac{w}{s})^w \le \frac{\binom{2s}{s}}{\binom{2s}{s-w}} \le (1+\frac{w}{s-w+1})^w$ or, equivalently, 
\begin{center}
$[(1+\frac{1}{s/w})^{s/w}]^{w^2/s} \le \frac{\binom{2s}{s}}{\binom{2s}{s-w}} \le [(1+\frac{1}{(s-w+1)/w})^{\frac{s-w+1}{w}}]^{\frac{w^2}{s-w+1}}$.
\end{center}
Since both lower and upper bounds approach $e^{w^2/s}$, the result follows.
\end{proof}
From Lemma \ref{L3} and Lemma \ref{L4}, we immediately get the following result.

\begin{corollary}\label{C5}
If $s = s(n) \xrightarrow [n \to\infty ]{}  \infty$ and $w(n) = o(s(n))$, then
\begin{align*}
\binom{s}{\frac{s}{2}-w} \sim \frac{\sqrt{2} \cdot 2^s}{\sqrt{\pi s} \cdot e^{\frac{2w^2}{s}}}.
\end{align*}
\end{corollary}

\begin{corollary}\label{C6}
The asymptotic solution of the inequality
	$\binom{n-2}{r - 1} \ge \frac{2^{n-2}}{n - 1}$ is $r \ge \frac{n}{2} - \frac{1}{2}\sqrt{n \ln \frac{2n}{\pi}}$ \rm{(}when considering only $r \le \frac{n}{2}$\rm{)}.
\end{corollary}
\begin{proof}
Let $s = n - 2$ and $r - 1 = \frac{s}{2}- w$.  We want to find $w$ for which\\
$\binom{s}{\frac{s}{2}-w} \ge \frac{2^s}{s+1}$. From Corollary \ref{C5}, we want to solve the inequality\\
 $$\frac{\sqrt{2} \cdot 2^s}{\sqrt{\pi s} \cdot e^{\frac{2w^2}{s}}} \ge \frac{2^s}{s+1} \text{  which is equivalent to   }e^{\frac{2w^2}{s}} \le \frac{\sqrt{2}(s+1)}{\sqrt{\pi s}} .$$
 We get $\frac{2w^2}{s} \le \ln \frac{\sqrt{2}(s+1)}{\sqrt{\pi s}} $ or, equivalently, $w \le \sqrt{\frac{s}{2}} \cdot (\ln \frac{\sqrt{2}(s+1)}{\sqrt{\pi s}})^{1/2} \sim \frac{1}{2}\sqrt{s \ln \frac{2s}{\pi}}$. \\
 Therefore,  $r = \frac{s}{2} + 1 - w \ge \frac{n}{2} - \frac{1}{2}\sqrt{n \ln \frac{2n}{\pi}}.$
\end{proof}

The rest of this section covers calculating the exact asymptotics of the probability that $\mathcal{A}_n$ succeeds. First, we define algorithm $\mathcal{A}(h_n, w_n)$ that is not optimal, but has more regular stopping region than the optimal algorithm $ \mathcal{A}_n$. It will be helpful in finding the reasonable lower bound for the performance of $\mathcal{A}_n$.

The algorithm $\mathcal{A}(h_n, w_n)$ takes natural parameters $h_n$ and $w_n$ describing its stopping region. It never stops before time $h_n$. For $t \ge h_n$ it stops on $x_t$ if and only if $x_t$ falls between $y_{\lceil \frac{t}{2}\rceil - w_n}^{(t-1)}$ and $y_{\lfloor \frac{t}{2} \rfloor +w_n}^{(t-1)}$, where $y_1^{(t-1)} < y_2^{(t-1)}  < ... < y_{t-1}^{(t-1)} $ are ordered numbers revealed till time $t-1$. If it never happens, $\mathcal{A}(h_n, w_n)$ stops at $x_n$.

\begin{algorithm}[H]
\DontPrintSemicolon
\caption{\textbf{$\mathcal{A}(h_n, w_n)$}}
\vspace{5pt}
\KwData{$x_1, x_2, \ldots, x_n$ - numbers chosen uniformly at random from~$[0,1]$; $h_n$, $w_n$ - parameters describing the stopping region} 
\KwResult{candidate for the number being the closest one to $1/2$ among $x_1, x_2, \ldots, x_n$}
\Begin{
	\For{$t = 1, 2, \ldots, h_n-1$}{
		reveal $x_t$\;
	}
	\For{$t = h_n, h_n+1, \ldots, n-1$}{
		reveal $x_t$\;
		\If { $x_t \in [y_{\lceil \frac{t}{2}\rceil-w_n}^{(t-1)}, y_{\lfloor \frac{t}{2} \rfloor+w_n }^{(t-1)}]$ }
		{\Return $x_t$\;}
	}
	\Return $x_n$\;
}
\vspace{5pt}
\footnotesize{Recall that  $y_1^{(t-1)} < y_2^{(t-1)}  < ... < y_{t-1}^{(t-1)} $ are ordered $x_1, x_2, \ldots, x_{t-1}$, i.e., ordered numbers revealed till round $t-1$.}
\end{algorithm}


\vspace{10pt}
Figure 2 displays the rectangular stopping region for the algorithm $ \mathcal{A}(h_n, w_n)$. Note that $n-h_n+1$ and $2 w_n$ can be interpreted as height and width of this stopping region, respectively.

\begin{figure}
\vspace*{80pt}
\begin{center}
\begin{picture}(270,120)


\put(0,0){\line(3,5){120}}
\put(240,0){\line(-3,5){120}}
\put(0,0){\line(1,0){240}}
\put(40,70){\line(1,0){4}}
\put(16,30){\line(1,0){4}}
\put(4,10){\line(1,0){4}}

\multiput(48,80)(5,0){29}{\line(2,0){2}}
\linethickness{.8mm}{\color{red}
{\multiput(90,70)(0,-10){3}{\line(1,0){60}}
\multiput(90,30)(0,-20){2}{\line(1,0){60}}}
\put(0,0){\line(1,0){240}}}							
\multiput(88,67)(60,0){2}{$\bullet$}
\multiput(88,57)(60,0){2}{$\bullet$}
\multiput(88,47)(60,0){2}{$\bullet$}
\multiput(88,27)(60,0){2}{$\bullet$}
\multiput(88,67)(60,0){2}{$\bullet$}
\multiput(88,7)(60,0){2}{$\bullet$}
\multiput(-3,-3)(240,0){2}{$\bullet$}   
\multiput(120,43)(0,-3){3}{\circle{1}}
\multiput(120,22)(0,-3){3}{\circle{1}}

\put(78,110){$\text{DO NOT STOP}$}
\put(13,80){$\small{h_n-1}$}
\put(23,66){$\small{h_n}$}
\put(9,28){$\small{t}$}
\put(-25,8){$\small{n-1}$}
\put(-15,-2){$\small{n}$}

\put(108,-12){width}
\put(106,-22){$\sim 2w_n$}
\put(90, -17){$\leftarrow$}
\put(141, -17){$\rightarrow$}

\end{picture}
\vspace*{1.5pc}

 {\bf Figure 2.} The stopping region for the algorithm $\mathcal{A}(h_n,w_n)$.
\end{center}
\end{figure}

The following technical lemma will be used in the next theorem to estimate the performance of $ \mathcal{A}(h_n, w_n)$.

\begin{lemma} \label{lemma_dec}
Let $a_s = \frac{1}{2^s} \binom{s}{\frac{s}{2}-w}$. For $ m > 2 w^2-1$, the sequences $\{a_{2m}\}_{m \geq 0}$ and $\{a_{2m+1}\}_{m \geq 0}$ are decreasing.
\end{lemma}

\begin{proof}
Consider the sequence $\{a_{2m+1}\}_{m \geq 0}$. Using the fact $\Gamma(z+1) = z \Gamma(z)$, we get
\begin{align*} 
\begin{split}
\frac{a_{2m+3}}{a_{2m+1}} & = \frac{2^{2m+1}}{2^{2m+3}} \cdot \frac{\binom{2m+3}{\frac{2m+3}{2}-w}}{\binom{2m+1}{\frac{2m+1}{2}-w}} \\
& = \frac{1}{4} \cdot \frac{\Gamma(2m+4)}{\Gamma(m-w+\frac{5}{2})\Gamma(m+w+\frac{5}{2})} \cdot \frac{\Gamma(m-w+\frac{3}{2}) \Gamma(m+w+\frac{3}{2})}{\Gamma(2m+2)} \\
& = \frac{1}{4} \cdot \frac{(2m+3)(2m+2)}{(m-w+\frac{3}{2})(m+w+\frac{3}{2})} = \frac{2m^2+5m+3}{2m^2+6m-2w^2+\frac{9}{2}}.
\end{split}
\end{align*}
Thus if only $m>2w^2-\frac{3}{2}$, the ratio $\frac{a_{2m+3}}{a_{2m+1}}$ is smaller than $1$ and the sequence $\{a_{2m+1}\}_{m \geq 0}$ is decreasing.

Similarly, we get that for $m>2w^2-1$, the sequence $\{a_{2m}\}_{m \geq 0}$ is decreasing.
\end{proof}

\begin{corollary} \label{cor_dec}
Let $s$, $w$, and $n$ be natural numbers such that $n > s$ and $s > 4 w^2$. Then
$$
\frac{1}{2^{s-1}}\binom{s-1}{\lceil \frac{s}{2} \rceil-w} \geq \frac{1}{2^{n-1}}\binom{n-1}{\frac{n-1}{2}-w}.
$$
\end{corollary}
\begin{proof}
Whenever $n$ and $s$ are both odd or both even, by Lemma \ref{lemma_dec} we get immediately
\begin{align*}
\frac{1}{2^{s-1}}\binom{s-1}{\lceil \frac{s}{2} \rceil-w} & \geq \frac{1}{2^{s-1}}\binom{s-1}{\frac{s-1}{2}-w} \geq \frac{1}{2^{n-1}}\binom{n-1}{\frac{n-1}{2}-w}.
\end{align*}
Now, assume that $s$ is even and $n$ is odd. Then, by Lemma \ref{lemma_dec}, since $n>s$
$$
\frac{1}{2^{s-1}} \binom{s-1}{\lceil \frac{s}{2} \rceil -w} = \frac{1}{2^{s-1}} \binom{s-1}{\frac{s}{2}-w} \geq \frac{1}{2^{s-1}} \cdot \frac{1}{2} \cdot \binom{s}{\frac{s}{2}-w} \geq \frac{1}{2^{n-1}}\binom{n-1}{\frac{n-1}{2}-w}.
$$
Finally, assume that $s$ is odd and $n$ is even. Then, analogously
\begin{align*}
\begin{split}
\frac{1}{2^{s-1}} \binom{s-1}{\lceil \frac{s}{2} \rceil -w} & = \frac{1}{2^{s-1}} \binom{s-1}{\frac{s+1}{2} -w} \geq \frac{1}{2^{s-1}} \cdot \frac{1}{2} \cdot \binom{s}{\frac{s+1}{2} - w} \\
& \geq \frac{1}{2^{s}} \cdot \binom{s}{\frac{s}{2}-w} \geq \frac{1}{2^{n-1}}\binom{n-1}{\frac{n-1}{2}-w}.
\end{split}
\end{align*}
\end{proof}

\begin{theorem}\label{th_main}
For sequences $h_n$ and $ w_n$ of natural numbers such that $h_n \leq n$ and $4w_n^2<n$, we have
$$\text{Pr}( \mathcal{A}(h_n, w_n)  \text{ succeeds)} \ge v(h_n,w_n),$$
where $v(h_n,w_n)$ is a function such that for $w_n \xrightarrow [n \to\infty ]{}  \infty$
$$ v(h_n, w_n) \sim \frac{h_n}{n2^{n-1 }} \binom{n-1}{\frac{n-1}{2}-w_n}\cdot \bigg(1 -  \Big(1 - \frac{2w_n}{h_n}\Big)^{n-h_n}\bigg).$$
\end{theorem}
\begin{proof}
For $s \in \{1,2,\ldots, n\}$ let $B_s$ be the event that the best element arrives at time $s$; it means that $x_s$ is closest to $\frac{1}{2}$.  Of course, $\text{Pr}(B_s) = \frac{1}{n}$.  Then
\begin{align*}
\begin{split}
\text{Pr}( \mathcal{A}(h_n, w_n) \text{ succeeds)} &= \sum\limits_{s=1}^{n}\text{Pr}( \mathcal{A}(h_n, w_n) \text{ succeeds }\big| B_s)\cdot  \text{Pr}(B_s) \\
& = \frac{1}{n}\sum\limits_{s=h_n}^{n}\text{Pr}( \mathcal{A}(h_n, w_n) \text{ succeeds }\big| B_s),
\end{split}
\end{align*}
because our algorithm never stops before time $h_n$. 

In order for $ \mathcal{A}(h_n, w_n)$ to succeed, the numbers $x_{h_n}, x_{h_n+1}, ..., x_{s-1}$ must fall outside the stopping intervals and the number $x_s$ must fall into the interval 
$\big[y_{\lceil\frac{s}{2}\rceil-w_n}^{(s-1)}, y_{\lfloor\frac{s}{2}\rfloor+w_n}^{(s-1)}\big]$.  These events are independent and
\begin{center}
$\text{Pr}\big( x_t \text{ falls outside } \big[y_{\lceil\frac{t}{2}\rceil-w_n}^{(t-1)}, y_{\lfloor\frac{t}{2}\rfloor+w_n }^{(t-1)}\big] \big| B_s \big) = 1 - \frac{1}{t}\big(\lfloor\frac{t}{2}\rfloor+w_n - \lceil\frac{t}{2}\rceil + w_n\big),$
\end{center}
 because the expression in parenthesis (which is at most $2w_n$) counts the number of intervals (out of $t$ intervals) that are forbidden for $x_t$.
This justifies that
\begin{equation}\label{falls_out}
\text{Pr}\big( x_t \text{ falls outside } \big[y_{\lceil\frac{t}{2}\rceil-w_n}^{(t-1)}, y_{\lfloor\frac{t}{2}\rfloor+w_n }^{(t-1)}\big] \big| B_s \big) \geq 1-\frac{2w_n}{t},
\end{equation}
 For the event that $x_s$ falls into the interval $ \big[y_{\lceil\frac{s}{2}\rceil-w_n }^{(s-1)}, y_{\lfloor\frac{s}{2}\rfloor+w_n}^{(s-1)}\big]$  given $B_s$ to happen, we can observe that $s - 1$ numbers that came before $x_s$ must all be outside the interval $ \big[\frac{1}{2}-d, \frac{1}{2}+d \big]$, where $d = \big|x_s - \frac{1}{2}\big|$.
 In order for $x_s$ to have a proper rank, such that the algorithm will stop on it, we must have $j$ numbers out of $s-1$ to fall into the left interval $ \big[ 0, \frac{1}{2}-d \big),$ where $j$ satisfies the inequality
  $\lceil\frac{s}{2}\rceil-w_n \le j \le \lfloor \frac{s}{2}\rfloor + w_n - 1.$  Therefore,
\begin{align} \label{falls_in}
\begin{split}
 \text{Pr}\big( x_s \text{ falls into } \big[y_{\lceil\frac{s}{2}\rceil-w_n }^{(s-1)}, y_{\lfloor\frac{s}{2}\rfloor+w_n}^{(s-1)}\big] \big| B_s \big) & = \sum\limits_{j=\lceil\frac{s}{2}\rceil-w_n}^{\lfloor \frac{s}{2} \rfloor +w_n - 1} \binom{s-1}{j}\frac{1}{2^{s-1}} \\
& \ge  \frac{1}{2^{s-1}}\cdot (2w_n-1) \cdot \binom{s-1}{\lceil \frac{s}{2}\rceil-w_n}.
\end{split}
\end{align}
Using inequalities (\ref{falls_out}) and (\ref{falls_in}), we get
\begin{align*}
\begin{split}
\text{Pr}\big( \mathcal{A}(h_n, w_n) \text{ succeeds}\big) & \ge \frac{1}{n}\sum\limits_{s=h_n}^{n} \left(\prod\limits_{t=h_n}^{s-1} \big( 1 - \frac{2w_n}{t}\big)\right)  \frac{2w_n-1}{2^{s-1}}  \binom{s-1}{\lceil \frac{s}{2} \rceil - w_n}  \\
&\ge  \frac{1}{n}  \frac{2w_n-1}{2^{n-1}}  \binom{n-1}{\frac{n-1}{2}-w_n}\sum\limits_{s=h_n}^{n}\big( 1 - \frac{2w_n}{h_n}\big)^{s-h_n} , 
\end{split}
\end{align*}
where the last inequality uses Corollary \ref{cor_dec} about monotonicity of $ \frac{1}{2^{s-1}}  \binom{s-1}{\lceil \frac{s}{2} \rceil-w_n} $.  After changing the index of summation in the last sum, we get
\begin{align*}
\begin{split}
\text{Pr}\big( \mathcal{A}(h_n, w_n) & \text{ succeeds}\big) \ge  \frac{2w_n-1}{n2^{n-1}}  \binom{n-1}{\frac{n-1}{2}-w_n}\sum\limits_{s=0}^{n-h_n}\Big( 1 - \frac{2w_n}{h_n}\Big)^{s} \\
&=  \frac{2w_n-1}{n2^{n-1}}  \binom{n-1}{ \frac{n-1}{2}-w_n } \frac{1}{\frac{2w_n}{h_n}}\bigg(1 - \Big( 1 - \frac{2w_n}{h_n}\Big)^{n-h_n+1}\bigg) \\ 
& = v(h_n, w_n).
\end{split}
\end{align*}
Whenever $w_n \xrightarrow [n \to\infty ]{}  \infty$, we get
$$
v(h_n, w_n) \sim \frac{h_n}{n2^{n-1}} \binom{n-1}{ \frac{n-1}{2}-w_n } \bigg(1 - \Big( 1 - \frac{2w_n}{h_n}\Big)^{n-h_n}\bigg).
$$
\end{proof}

\begin{corollary}\label{C9}
If $h_n = \lceil n\big( 1 - \frac{f(n)}{g(n)}\big) \rceil$ is such that $f(n) \xrightarrow [n \to\infty ]{}  \infty$, \\
$f(n) = o(g(n))$ and integer sequence $w_n$ satisfies $w_n \xrightarrow [n \to\infty ]{}  \infty $, $ w_n = o(\sqrt{h_n})$, and $\frac{g(n)}{w_n} = o(f(n))$, then
$$\sqrt{n}~\text{Pr}\big( \mathcal{A}(h_n, w_n) \text{ succeeds}\big) \ge \sqrt{n}~v(h_n,w_n) \xrightarrow [n \to\infty ]{} \sqrt{\frac{2}{\pi}},$$
where inequality holds for sufficiently large $n$.
\end{corollary}
\begin{proof}
From Theorem \ref{th_main}, we get for sufficiently large $n$
$$\sqrt{n}~\text{Pr}\big( \mathcal{A}(h_n, w_n) \text{ succeeds}\big) \ge \sqrt{n}~v(h_n, w_n)$$
and
$$\sqrt{n}~v(h_n, w_n) \sim \sqrt{n}  \frac{h_n}{n2^{n-1}} \binom{n-1}{\frac{n-1}{2}-w_n} \bigg(1 - \Big( 1 - \frac{2w_n}{h_n}\Big)^{n-h_n}\bigg). $$
The last factor can be rewritten as \\
$$1 - \bigg[\Big( 1 - \frac{1}{\frac{h_n}{2w_n}}\Big)^{\frac{h_n}{2w_n}}\bigg]^{\frac{2w_n(n-h_n)}{h_n}}$$
and since $\frac{h_n}{2w_n}\xrightarrow [n \to\infty ]{}  \infty$ and $\frac{2w_n(n-h_n)}{h_n} = \frac{2w_n\Big(n- \big(n - \frac{f(n)}{g(n)}n\big)\Big)}{ \big(1 - \frac{f(n)}{g(n)}\big)n} =  \frac{2w_n \frac{f(n)}{g(n)}}{1 - \frac{f(n)}{g(n)}} \sim$ \\
$$ \sim  2w_n \frac{f(n)}{g(n)} = \frac{2f(n)}{\frac{g(n)}{w_n}} \xrightarrow [n \to\infty ]{}  \infty, \text{ because } \frac{g(n)}{w_n} = o(f(n)),$$
the last factor approaches $1$ as $n \to \infty$. \\
For the remaining factors, using Corollary \ref{C5}, we have
\begin{align*}
\begin{split}
\frac{\sqrt{n}  h_n}{n2^{n-1}} \binom{n-1}{\frac{n-1}{2}-w_n} &\sim \frac{\sqrt{n}  \big(1- \frac{f(n)}{g(n)} \big)}{2^{n-1}} \frac{\sqrt{2} \cdot 2^{n-1}}{\sqrt{\pi(n-1)}\cdot e^{2w_n^2/(n-1)}} \\
&\sim \sqrt{\frac{2}{\pi}} \frac{1}{e^{2w_n^2/n}} \sim \sqrt{\frac{2}{\pi}},
\end{split}
\end{align*}
because $w_n = o(\sqrt{n})$ and the result follows.
\end{proof}

There are choices of sequences $h_n$ and $w_n$ satisfying assumptions of Corollary \ref{C9}, for example $w_n = \lceil n^{1/3} \rceil$ and $h_n = \lceil n \big( 1 - \frac{\sqrt{\ln n}}{n^{1/3}}\big) \rceil$ (this choice of $h_n$ is not accidental, it equals $\lceil n-n^{2/3}\sqrt{\ln{n}} \rceil$, which is the simulated number of round till which algorithm $\mathcal{A}_n$ does not take any decision - consult the beginning of this section).  For these choices of the stopping region, the algorithm $\mathcal{A}(h_n, w_n) $ is bounded from below by the function which asymptotically behaves as $\frac{1}{\sqrt{n}}\sqrt{\frac{2}{\pi}}$. 
Since optimal algorithm $\mathcal{A}_n$ is not worse, this lower bound applies also to $\mathcal{A}_n$.
It remains to prove that the asymptotic upper bound for the performance of $\mathcal{A}_n$ is the same.

\begin{theorem}\label{}
For the online decision problem of stopping on the number closest to $\frac{1}{2}$ with $n$ numbers coming randomly from the interval $[0, 1]$ with the knowledge of their ranks only, the optimal stopping algorithm 
$\mathcal{A}_n$ has asymptotic performance
$$\text{Pr}\big( \mathcal{A}_n \text{ succeeds} \big) \sim \frac{1}{\sqrt{n}} \sqrt{\frac{2}{\pi}}.$$
\end{theorem}

\begin{proof}
From the analysis of the algorithm $\mathcal{A}(h_n, w_n)$ (Corollary \ref{C9}), we know that the performance of $\mathcal{A}_n$ may be bounded from below by the function which asymptotically behaves as $\frac{1}{\sqrt{n}}\sqrt{\frac{2}{\pi}}$.


To find the upper bound, we consider an online decision problem that is much easier than the problem in question.
Suppose that $n$ numbers from the interval $[0, 1]$ are revealed one by one and we know their relative ranks at any time $t$, $1 \le t \le n$.  After revealing all $n$ numbers, we can select any number we like, not necessarily the last one. Our aim is still the same: maximizing the probability of choosing the element which is closest to $\frac{1}{2}$. Then the optimal strategy is simple. From Theorem \ref{bestprob}, we know that we have to select the number of rank $r$ such that the binomial coefficient $\binom{n-1}{r-1}$ has maximum value. This happens if 
$r - 1 = \lfloor \frac{n-1}{2} \rfloor$ or $r -1  = \lceil \frac{n-1}{2} \rceil$.  Then 
$$ \text{Pr}\big( x_r \text{ is the best } \big) = \binom{n-1}{\lfloor \frac{n-1}{2} \rfloor} \cdot \frac{1}{2^{n-1}} $$
and using Lemma \ref{L3}, we get
$$ \text{Pr}\big( x_r \text{ is the best } \big) \sim \binom{n-1}{ \frac{n-1}{2}} \cdot \frac{1}{2^{n-1}} \sim \frac{\sqrt{2}\cdot 2^{n-1}}{\sqrt{\pi n}} \cdot \frac{1}{2^{n-1}} =  \frac{1}{\sqrt{n}} \sqrt{\frac{2}{\pi}},$$
which gives the asymptotic upper bound for the performance of $\mathcal{A}_n$.
\end{proof}

In Figures 3 and 4, we present the asymptotic behaviour of the performance of $\mathcal{A}_n$. We have there
$$
\tilde{v}(h_n, w_n) = \frac{h_n}{n2^{n-1 }} \binom{n-1}{\frac{n-1}{2}-w_n}\cdot \bigg(1 -  \Big(1 - \frac{2w_n}{h_n}\Big)^{n-h_n}\bigg)
$$
which is a function from Theorem \ref{th_main} reflecting the asymptotic behaviour of the function $v(h_n, w_n)$.
The choices of $h_n$ and $w_n$ are $h_n = \lceil n \big( 1 - \frac{\sqrt{\ln n}}{n^{1/3}}\big) \rceil$ and $w_n = \lceil n^{1/3} \rceil$.

\vspace{10pt}

\begin{figure}[!ht]
\begin{center}
\includegraphics[width = 0.8 \textwidth]{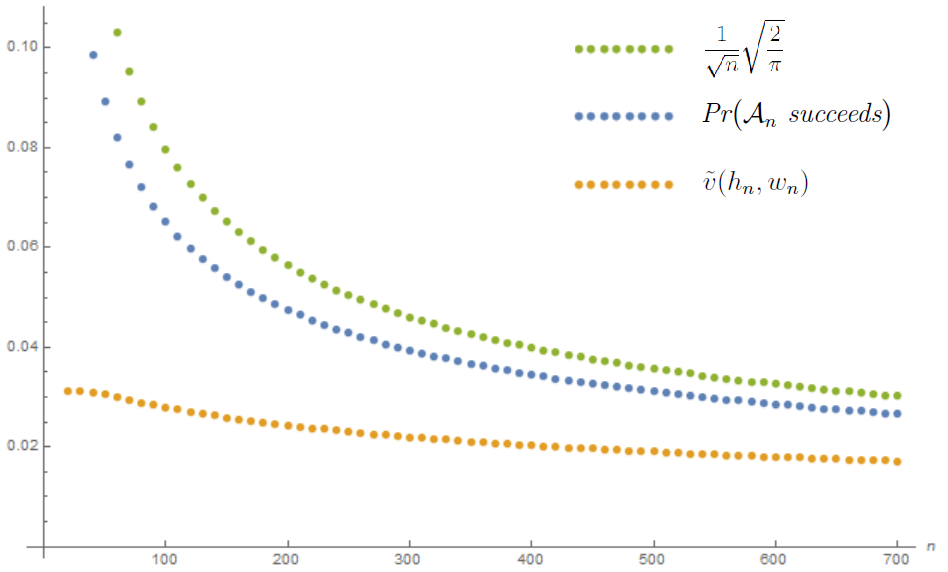} \\
\textbf{Figure 3.} Performance of $\mathcal{A}_n$ with its asymptotic bounds for $h_n = \lceil n \big( 1 - \frac{\sqrt{\ln n}}{n^{1/3}}\big) \rceil$ and $w_n = \lceil n^{1/3} \rceil$.
\end{center}
\end{figure}

\begin{figure}[!ht]
\begin{center}
\includegraphics[width = 0.8 \textwidth]{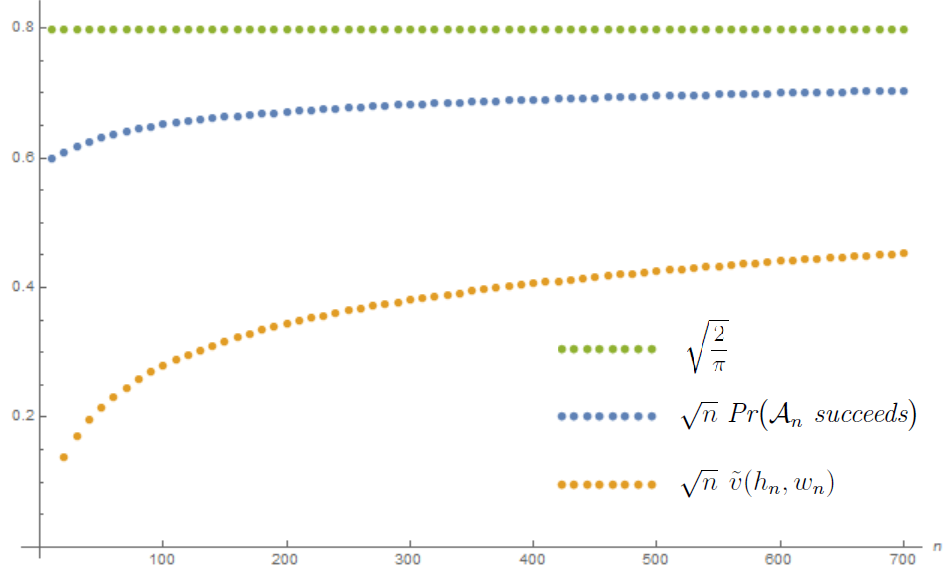}\\
\textbf{Figure 4.} Function $\sqrt{n}~\text{Pr}\big( \mathcal{A}_n \text{ succeeds}\big)$ with its asymptotic bounds for $h_n = \lceil n \big( 1 - \frac{\sqrt{\ln n}}{n^{1/3}}\big) \rceil$ and $w_n = \lceil n^{1/3} \rceil$.
\end{center}
\end{figure}	

\section{Final remarks}

If the interval $[0,1]$ is replaced by an interval $[a,b]$, with $a < b$, and the goal is to stop on the element closest to its midpoint, then the optimal stopping algorithm is identical to our algorithm $\mathcal{A}_n$.

How the situation changes if we sequentially observe $n$ numbers from the interval $[0,1]$, but we are informed about the value of each number drawn? Since we now know whether the revealed number is smaller or greater than $\frac{1}{2}$, by replacing each $x_k$ greater then $\frac{1}{2}$ by $1 - x_k$, we get the problem equivalent to finding the maximum element of the sequence of $n$ numbers. This problem was solved by Gilbert and Mosteller \cite{GM} and the optimal strategy in the process they called 'the full-information game' has asymptotic performance around 0.580164.  If we were interested in optimal stopping algorithm that would minimize the expected difference between selected number and $\frac{1}{2}$, then we could also adopt another stopping algorithm from \cite{GM} whose asymptotic performance is of order $\frac{1}{n}$.

\vspace{.5cm}

\noindent {\footnotesize \textbf{Ewa M. Kubicka}, University of Louisville, USA\\
\textbf{Grzegorz Kubicki}, University of Louisville, USA \\ 
\textbf{Ma{\l}gorzata Kuchta}, Wroc{\l}aw University of Science and Technology, Faculty of Fundamental Problems of Technology, Department of Computer Science, Poland\\
\textbf{Ma{\l}gorzata Sulkowska}, Wroc{\l}aw University of Science and Technology, Faculty of Fundamental Problems of Technology, Department of Computer Science, Poland}\\
\\

\noindent {\footnotesize e-mail addresses: 
ewa@louisville.edu,
gkubicki@louisville.edu,
malgorzata.kuchta@pwr.edu.pl, 
malgorzata.sulkowska@pwr.edu.pl}


\begin{thebibliography}{20} 


\bibitem{BS}
F.S. Benevides and M. Sulkowska, Percolation and best choice problem for powers of paths, J. Appl. Probab. {\bf 54} (2017), no. 2, pp. 343--362.

\bibitem{F} T.S. Ferguson,  Who solved the secretary problem? Statist. Sci. {\bf 4} (1989), pp. 282--296.

\bibitem{FW} R. Freij and J. W\"astlund, Partially ordered secretaries, Electron. Commun. Probab. {\bf 115} (2010), pp. 504--507.

\bibitem{GKMN}
N. Georgiou, M. Kuchta, M. Morayne and J. Niemiec,  On a universal best choice algorithm for partially ordered sets, Random Struct. Algor. {\bf 32} (2008), pp. 263--273.

\bibitem{GM}
J.P. Gilbert and F. Mosteller,  Recognizing the maximum of a sequence, J. Amer. Statist. Assoc. {\bf 61} (1966),  pp. 35--73.

\bibitem{GKK}
G. Goddard, E. Kubicka and G. Kubicki, An efficient algorithm for stopping on a sink in a directed graph, Oper. Res. Lett. {\bf 41} (2013), pp. 238--240.

\bibitem{KM}
G. Kubicki and M. Morayne, Graph-theoretic generalization of the secretary problem: the directed path case, SIAM J. Discrete Math. {\bf 19} (2005), no. 3, pp. 622--632.


\bibitem{L}
D.V. Lindley, Dynamic programming and decision theory, Appl. Statist. {\bf 10} (1961), pp. 39--51.

\bibitem{M}
M. Morayne, Partial-order analogue of the secretary problem; the binary tree case, Discrete Math. {\bf 184} (1998), pp. 165--181.

\bibitem{P}
J. Preater, The best-choice problem for partially ordered objects, Oper. Res. Lett. {\bf 25} (1999), pp. 187--190.

\bibitem{R}
P.A. Rogerson, Probabilities of choosing applicants of arbitrary rank in the secretary problem, J. Appl. Probab. {\bf 224} (1987), no. 2, pp. 527--533.

\bibitem{S}
W. Stadje, Efficient stopping of a random series of partially ordered points, Lecture Notes in Econom. and Math. Systems {\bf 177} (1980), pp. 430--447.

\bibitem{Su}
M. Sulkowska, The best choice problem for upward directed graphs, Discrete Optimization {\bf 9} (2012), pp. 200--204.


\end{thebibliography}
\end{document}